\documentclass{amsart}
\usepackage{bbm, enumerate}

\def\E{{\mathbb{E}}}
\def\P{{\mathbb{P}}}
\def\R{{\mathbb{R}}}

\def\Bcal{{\mathcal{B}}}
\def\Ccal{{\mathcal{C}}}
\def\Ecal{{\mathcal{E}}}

\def\Gcal{{\mathcal{G}}}
\def\Rcal{{\mathcal{R}}}
\def\Tcal{{\mathcal{T}}}
\def\Vcal{{\mathcal{V}}}

\newcommand{\indicator}[1]{\mathbbm{1}_{\{#1\}}}

\newtheorem{theorem}{Theorem}
\newtheorem*{proposition}{Proposition}
\newtheorem{lemma}{Lemma}

\theoremstyle{definition}
\newtheorem*{definition}{Definition}
\newtheorem{remark}{Remark}

\title{Coupling limit order books and branching random walks}

\author{Florian Simatos}
\thanks{This work was done while the author was affiliated with CWI and sponsored by an NWO-VIDI grant.}
\address{Eindhoven University of Technology}
\email{f.simatos@tue.nl}

\date{\today}

\begin{document}

\maketitle

\begin{abstract}
	We consider a model for a one-sided limit order book proposed by Lakner et al.~\cite{Lakner:0}. We show that it can be coupled with a branching random walk and use this coupling to answer a non-trivial question about the long-term behavior of the price. The coupling relies on a classical idea of enriching the state-space by artificially creating a filiation, in this context between orders of the book, that we believe has the potential of being useful for a broader class of models.
\end{abstract}

\section{Introduction}

\textbf{Limit order books.} A limit order book is a financial trading mechanism that keeps track of orders made by traders, and allows them to execute trades in the future. Typically, a trader places an order to buy a security at a certain level $x$. If the price of the security is larger than $x$ when the order is placed, then the order is kept in the book and may be fulfilled later in the future, as the price of the security fluctuates and falls below $x$. Similarly, traders may place sell orders, which gives rise to two-sided order books. Because of the importance of limit order books in financial markets, there has been a lot of research on these models, see for instance the survey by Gould et al.~\cite{Gould12:0}.

There are many variants of information of the book which traders have access to. For instance, traders may only have access to the current so-called bid and ask prices, that correspond to the lowest sell order and the highest buy order. In this case, traders have an incentive to place orders in the vicinity of these prices. More generally, the dynamic of a limit order book is intricate because its current state influences its future evolution. Stochastic models capturing this dynamic have for instance been proposed in Cont et al.~\cite{Cont10:0}, Lakner et al.~\cite{Lakner:0} and Yudovina~\cite{Yudovina:0}. In the present paper we study the one-sided limit order book model of Lakner et al.~\cite{Lakner:0}, and our goal is to show how some properties of this model can be efficiently studied thanks to a coupling with a branching random walk.

From a high-level perspective, the coupling we introduce adds a new dimension to the initial limit order book model by creating a filiation between the orders. Such ideas have been extremely successful in queueing theory, see for instance Kendall~\cite{Kendall51:0}, and we believe they can also be useful beyond the context of the model proposed here. For instance, the model proposed by Yudovina~\cite{Yudovina:0} is also amenable to a tree representation, but the corresponding dynamic on trees is more challenging to analyze than the one here.
\\

\textbf{Branching random walks.} The Galton Watson process is the simplest model of a branching process. It represents the size of a population that evolves in discrete time, where at every time step each individual dies and is replaced by a random number of offspring, see for instance Athreya and Ney~\cite{Athreya04:0} for more details. A branching random walk is an extension of a Galton Watson process that adds a spatial component to the model. In addition to the genealogical structure given by the Galton Watson process, each individual has some location, say on the real line $\R$, that is given by a random displacement of her parent's location. Branching random walks can therefore be represented by trees with labels on the edges: the structure of the tree represents the genealogy of the underlying Galton Watson process, and the labels on the edges represent the displacement of the child with respect to her parent's location. In this paper we will consider the simplest model of branching random walks, where labels on the edges are i.i.d., and will use results by Biggins~\cite{Biggins76:0} and Biggins et al.~\cite{Biggins91:0} to study the limit order book model.
\\

\textbf{Acknowledgements} The author would like Josh Reed for introducing him to the limit order book model of the present paper, Elie A\"id\'ekon for interesting discussions on branching random walks and an anonymous referee whose careful reading led to substantial improvements in the proof of Theorem~\ref{thm:price}.

\section{One-sided limit order book model}

\subsection{Model} Let us define a \emph{book} as a finite point measure on $\R$ and an \emph{order} as a point of a book. Let $\Bcal$ be the set of books. For a book $\beta \in \Bcal$, let $|\beta|$ be its mass (i.e., the number of orders it contains) and $\pi(\beta)$, which we call the \emph{price} of the book, be the right endpoint of its support, i.e., the location of the rightmost order:
\[ \pi(\beta) = \max \left\{ x \in \R: \beta(\{x\}) > 0 \right\}, \ \beta \in \Bcal. \]

When the book is empty, i.e., $|\beta| = 0$, the value of the price is inconsequent for the purposes of this paper, say for instance $\pi(\beta) = 0$. Fix $p \in [0,1]$ and a real-valued random variable $X$. We are interested in the $\Bcal$-valued Markov chain $(B_n, n \geq 0)$ with the following dynamic. 

If the book is empty, the process goes to state $\delta_0$ in the next time step, where here and in the sequel $\delta_x$ stands for the Dirac mass at $x \in \R$. If the book is not empty, a coin with bias $p$ is flipped. If heads (with probability $p$), an order is added to the current book at a random distance distributed according to $X$ from the current price of the book, independently from everything else. If tails (thus, with probability $1-p$), an order sitting at the current price is removed. Formally, the Markov chain $(B_n)$ has the following dynamic: for any $\beta \in \Bcal$ and any measurable function $f: \Bcal \to [0,\infty)$,
\begin{multline} \label{eq:transition}
	\E\left[f(B_{n+1}) \mid B_n = \beta\right] = f(\delta_0) \indicator{|\beta| = 0} + p \E \left[f(\beta + \delta_{\pi(\beta) + X}) \right] \indicator{|\beta| > 0}\\
	+ (1-p) f(\beta - \delta_{\pi(\beta)}) \indicator{|\beta| > 0}.
\end{multline}

As explained in the introduction, this model (in continuous-time, and with a different boundary condition) has been proposed by Lakner et al.~\cite{Lakner:0} to model a one-sided limit order book. The interpretation of the model is as follows: $B_n$ represents the state of a one-sided limit order book with only buy orders. In each time step, either a trader places a new buy order (with probability $p$), or a trader places a market order (to sell the security, with probability $1-p$). In the latter case, the trader sells the security at the highest available buy order, thus removing one order sitting at the price. From this perspective, the behavior of the price process $(\pi(B_n), n \geq 0)$ is of primary interest. In this paper, we show how one can answer questions related to the price process by coupling $(B_n)$ with a branching random walk. Our coupling can be used to answer more elaborate questions on this particular model, and we believe that it also has the potential to be applied to other models.

\begin{remark}
	In order to recover the model of Lakner et al.~\cite{Lakner:0} one needs to apply an exponential transformation to $(B_n)$, i.e., to consider the process $\widetilde B_n = \sum_{x \in B_n} \delta_{e^x}$. This transformation makes the orders live on $(0,\infty)$, in which case the interpretation of $\pi$ as a price is reasonable. It also makes the price increase in a multiplicative rather than linear fashion, which is a common behavior in mathematical finance (e.g., geometric Brownian motion). We prefer to state the model on the line with linear displacement because of the analogy with branching random walks.
\end{remark}

\subsection{Price dynamic} The behavior of the price is asymmetric due to the system's dynamic. On the one hand, the price increases when an order is added to the right of the current price, and so an increase of the price is distributed according to $X$ given that $X > 0$. On the other hand, the price decreases when an order is removed from the book, in which case the decrease of the price depends on the distance between the price and the second rightmost order. In particular, orders to the left of the price act as a barrier that slow down the price as it wants to drift downward.

Thus, although $\E X < 0$ seems at first a natural condition for the price to drift to $-\infty$, it seems plausible that if $p$ is sufficiently close to $1$, there will be so many orders sitting to the left of the price that they will eventually make the price drift to $+\infty$. This intuition turns out to be correct as Theorem~\ref{thm:price} below shows.

This kind of behavior is strongly reminiscent of the behavior of extremal particles in branching random walks. There, although a typical particle drifts to $-\infty$ when $\E X < 0$, one may still observe atypical trajectories due to the exponential explosion in the number of particles, see the classical references by Hammersley~\cite{Hammersley74:0}, Kingman~\cite{Kingman75:1} and Biggins~\cite{Biggins76:0}. This analogy has actually been our initial motivation to investigate the relation between $(B_n)$ and branching random walks. And indeed, we will show in Theorem~\ref{thm:coupling} that $B_n$ can be realized as some functional of a branching random walk, and this connection will make the proof of Theorem~\ref{thm:price} quite intuitive.

\begin{theorem} \label{thm:price}
	Assume that $p > 1/2$ and that $\E X$ exists in $(-\infty, \infty)$.

	If $\E X > 0$, then $\pi(B_n) \to +\infty$ almost surely.

	Else, assume in addition to $p > 1/2$ that $\E X < 0$ and that $\P(X > 0) > 0$, and let $a = \inf_{\theta \geq 0} \E(e^{\theta X}) \in (0,1]$. If $p > 1/(1+a)$, then $\pi(B_n) \to +\infty$ almost surely, while if $p < 1/(1+a)$ then $\pi(B_n) \to -\infty$ almost surely.
\end{theorem}

\begin{remark} \label{rem:a}
	Let $\varphi(\theta) = \E(e^{\theta X})$ for $\theta \geq 0$ and assume that $\E X < 0$. We will use the following dichotomy: either $\varphi(\theta) = +\infty$ for every $\theta > 0$, in which case $a = 1$; or $\varphi(\theta) < +\infty$ for some $\theta > 0$, in which case $a < 1$ due to the fact that $\varphi'(0) = \E X < 0$ in this case.
\end{remark}

Note that the price process is recurrent if $p \leq 1/2$ since $(|B_n|, n \geq 0)$ is a random walk reflected at $0$, so the previous result gives a full picture of the price behavior (except for the boundary cases $\E X = 0$ and $p = 1/(1+a)$). It is interesting to observe that if $p > 1/2$, $\E X < 0$, $\P(X > 0) > 0$ and $X$ has a heavy right tail, in the sense that $\varphi(\theta) = +\infty$ for every $\theta > 0$ or equivalently, the random variable $\max(X, 0)$ has no finite exponential moment, then the price will always diverge to $+\infty$, irrespectively of the values of $p$ and $\E X$. Although the fact that exponential moments play a key role is clear from a branching process perspective, we find it more surprising from the perspective of the limit order book.

\section{Coupling with a branching random walk}

\subsection{Intuition} From the book process $(B_n)$, one can construct a genealogical structure by making an order $x$ a child of some other order $y$ if $x$ was added to the book at a time where $y$ corresponded to the price of the book, i.e., $x$ is added at a time $n$ where $y = \pi(B_n)$. Since there is each time a probability $p$ of adding an order to the book, it is intuitively clear that this construction will give each order (at most) a geometric number of offspring. By labeling the edge between $x$ and $y$ with the displacement $x-y$, which has distribution $X$, we end up with a Galton Watson tree with geometric offspring distribution and i.i.d.\ real-valued labels on the edges, i.e., a branching random walk. The idea of the coupling is to reverse this construction and to start from the branching random walk to build the book process $(B_n)$. To do so, we will essentially realize the process $(B_n)$ as the iteration of a deterministic tree operator $\Phi$ on a random tree, thus encoding all the randomness in the tree.

Nodes of the tree represent orders of the books, and in order to distinguish between orders that are currently in the book, orders that have been in the book and removed, and orders that have not been in the book so far (but may be later) we consider trees where nodes have one of three colors: green (orders currently in the book), red (orders removed from the book) and white (orders not added in the book so far). We also consider trees with real-valued labels on the edges: then, each node is also given a label by adding to the label of its parent the label on the edge between them, the root having some arbitrary real-valued label. The label of a node represents the position of the corresponding order in the book. Then, the green node with largest label, say $\gamma$, represents the order at the current price, and so we will run the following dynamic on trees:
\begin{itemize}
	\item if $\gamma$ has at least one white child, then its first white child becomes green;
	\item if $\gamma$ has no white child, then $\gamma$ becomes red;
	\item if the tree has no green node then we need to draw a new random tree.
\end{itemize}

\subsection{The coupling: notation and main result}

Let $\Tcal$ be the set of rooted trees where:
\begin{itemize}
	\item every edge has a real-valued label;
	\item every node has one of three colors, green, red or white;
	\item finitely many nodes are green or red, and the set of nodes that are either green or red is connected;
	\item the root is green or red.
\end{itemize}

We will need to compare the labels and colors of the nodes and edges of various trees. In that respect, it is convenient to consider $\Vcal$ the set of all possible nodes and $\Ecal$ the set of all possible edges, and to denote by $\Vcal(t) \subset \Vcal$ the set of nodes, $\Ecal(t) \subset \Ecal$ the set of edges and $\Gcal(t), \Rcal(t) \subset \Vcal(t)$ the set of green and red nodes, respectively, of a tree $t \in \Tcal$. Nodes inherit labels as explained above, i.e., the label of a node is obtained by adding to the label of its parent the label of the edge between them, and the root has any real-valued label. If $v \in \Vcal$ is a node and $e \in \Ecal$ is an edge, we denote by $\ell(v, t)$ and $\ell(e,t)$ the label of this node and edge in the tree $t \in \Tcal$, provided $v \in \Vcal(t)$ and $e \in \Ecal(t)$. We call genealogical structure of a tree $t \in \Tcal$ the tree obtained from $t$ when forgetting about labels and colors, and we say that $t$ is a subtree of $t'$ and write $t \subset t'$ if the genealogical structure of $t$ is a subtree of the genealogical structure of $t'$ (in the usual graph-theoretic sense) and $\ell(e, t) = \ell(e, t')$ for every $e \in \Ecal(t) \subset \Ecal(t')$. For $t \in \Tcal$ let $\Gamma(t) \in \Bcal$ be the point measure recording the labels of the green nodes of $t$:
\[ \Gamma(t) = \sum_{v \in \Gcal(t)} \delta_{\ell(v, t)}. \]

Let $\Tcal^* = \{t \in \Tcal: |\Gcal(t)| > 0\}$ be the set of trees with at least one green node. If $t \in \Tcal^*$ we denote by $\gamma(t)$ the green node with largest label and by $\omega(t)$ the number of white children of $\gamma(t)$. If there are several green nodes with maximal labels, we choose the last one (where in the sequel, nodes are ordered according to the lexicographical order). For $t \in \Tcal$ and $v \in \Vcal(t)$ we will more generally define $\omega(v,t)$ as the number of white children of $v$ in $t$, so that $\omega(t) = \omega(\gamma(t), t)$ for $t \in \Tcal^*$. The following operator will create the dynamic of $(B_n)$.
\begin{definition} [Operator $\Phi$]
	Let $\Phi: \Tcal \to \Tcal$ be the following operator: if $|\Gcal(t)| = 0$ then $\Phi(t) = t$, while if $t \in \Tcal^*$, then $\Phi$ changes the color of one node according to the following rule:
	\begin{itemize}
		\item if $\omega(t) > 0$, $\Phi$ transforms the first white child of $\gamma(t)$ into a green node;
		\item if $\omega(t) = 0$, then $\Phi$ transforms $\gamma(t)$ into a red node.
	\end{itemize}
	
	We also define $\Phi_n$ as the $n$th iterate of $\Phi$, defined by $\Phi_0$ being the identity map and $\Phi_{n+1} = \Phi \circ \Phi_n$ for $n \geq 0$. Finally, we define $\kappa(t) = \inf\{n \geq 0: |\Gcal(\Phi_n(t))| = 0 \}$ for $t \in \Tcal$, so that $\kappa(t) \in \{0,1,\ldots,\infty\}$ is the first time where iterating $\Phi$ on $t$ creates a tree with no green node.
\end{definition}

We can now state our main result.

\begin{theorem} \label{thm:coupling}
	Let $T$ be the following random tree:
	\begin{itemize}
		\item the genealogical structure of $T$ is a Galton Watson tree with geometric offspring distribution with parameter $p$;
		\item labels on the edges are i.i.d.\ with distribution $X$, independently from the genealogical structure, and the root has label $0$;
		\item all nodes are white, except for the root which is green.
	\end{itemize}
	
	Then $(\Gamma(\Phi_n(T)), 0 \leq n \leq \kappa(T))$ is equal in distribution to $(B_n, 0 \leq n \leq \tau)$ started at $B_0 = \delta_0$, where $\tau = \inf\{ n \geq 0: |B_n| = 0 \}$.
\end{theorem}

\subsection{Proof of Theorem~\ref{thm:price} based on Theorem~\ref{thm:coupling}}

Thanks to Theorem~\ref{thm:coupling}, we can write $B_n = \Gamma(\Phi_n(T))$ for $n \leq \kappa(T) = \tau$. It is not hard to show that $\kappa(T) = +\infty$ if and only if $T$ is infinite, and so we will call $\{ \tau = \kappa(T) = +\infty \}$ the event of non-extinction. Since we are in the supercritical case $p > 1/2$, this event occurs with positive probability.

\subsubsection{First case: $\E X < 0$, $\P(X > 0) > 0$ and $p < 1/(1+a)$.} We first consider the case where the price drifts to $-\infty$, i.e., we assume that $\E X < 0$, $\P(X > 0) > 0$ and $p < 1/(1+a)$ and we prove that $\pi(B_n) \to -\infty$. Since $p > 1/2$ by assumption, we have in particular $a < 1$ and so there must exist $\eta > 0$ such that $\E(e^{\eta X}) < +\infty$ (see Remark~\ref{rem:a} following Theorem~\ref{thm:price}). Moreover, since we are interested in the long-time behavior of the price which goes back to $0$ at time $\tau$ in the event $\{\tau < +\infty\}$, we work in the event of non-extinction.

Let $M_n$ be the rightmost point of the branching random walk $T$ at time $n$, i.e.,
\[ M_n = \max \left\{ \ell(v, T): v \in \Vcal(T) \text{ and } |v| = n \right\} \]
where $|v|$ is the distance from $v$ to the root. Under the assumptions made on $X$ and $p$, Theorem~$4$ in Biggins~\cite{Biggins76:0} shows that $M_n \to -\infty$ almost surely in the event of non-extinction. We now show that $\pi(B_n) \to -\infty$ when $M_n \to -\infty$, which will conclude the proof of this case.

Let $K \geq 0$ and $n_0$ such that $M_n \leq -K$ for any $n \geq n_0$, i.e., $\ell(v,T) \leq -K$ for every $v \in \Vcal(T)$ with $|v| \geq n_0$. Since there are only finite many nodes of $T$ at depth $<n_0$, we must have $|\gamma(\Phi_n(T))| \geq n_0$ for $n$ large enough and for those $n$, we have $\ell(\gamma(\Phi_n(T)), \Phi_n(T)) \leq -K$ by choice of $n_0$. This proves that $\pi(B_n) \to -\infty$.

\subsubsection{Second case: $\E X > 0$, or $\E X < 0$, $\P(X > 0) > 0$ and $p > 1/(1+a)$.} We now consider the case where the price drifts to $+\infty$, i.e.,  we assume that either $\E X > 0$, or $\E X < 0$, $\P(X > 0) > 0$ and $p > 1/(1+a)$ and we prove that $\pi(B_n) \to +\infty$. Before going into the technical details, let us give a high-level idea of the proof.

The key observation is that as long as an order is in the book, the behavior of the price does not depend on the state of the book to the left of this order. In particular, in order to compute the probability that the order sitting initially in the book at $0$ is never removed from the book, we may as well assume that all orders that are placed in $(-\infty, 0)$ are instantaneously removed, or killed.

In terms of the underlying tree $T$, removing all orders that are placed in $(-\infty, 0)$ amounts to removing all nodes $v$, together with all their descendants, with label $\ell(v, T) < 0$. We thus obtain a new tree $T'$, a subtree of the original tree $T$, which is a well-known object: this is precisely a branching random walk with a barrier at $0$. Under our assumptions on $X$ and $p$, the probability $\P(|T'| = +\infty)$ of $T'$ being infinite is strictly positive, as has been shown in Biggins et al.~\cite{Biggins91:0} (see Lemma~\ref{lemma:barrier} below).

Going back to the limit order book, $T'$ being infinite means exactly that the initial order sitting at $0$ will never be removed. Since this happens with positive probability, an order is eventually added to the book that is never removed. This order then constitutes a barrier under which the price never falls. Then, a renewal type argument shows that this phenomenon repeats itself: at regular intervals, an order is added to the book that constitutes a new barrier under which the price never falls. Eventually, this barrier moves up and forces to price to diverge to $+\infty$. Let us now formalize this heuristic argument.
\\

For $\ell \in \R$, we define the three operators $S_\ell, \Xi_\ell, \Xi: \Tcal \to \Tcal$ for $t \in \Tcal$ as follows:
\begin{itemize}
	\item $S_\ell(t)$ is obtained by adding $\ell$ to the label of the root;
	\item $\Xi_\ell(t)$ is obtained by removing all nodes $v \in \Vcal(t)$ with label $\ell(v, t) < \ell$, together with all their descendants;
	\item $\Xi(t) = \Xi_L(t)$ with $L$ the label of the root of $t$.
\end{itemize}

Note that since the label of each node is inherited from the label of its parent, adding $\ell$ to the label of the root has the effect of adding $\ell$ to the label of every node in the tree. Thus $S_\ell$ can be seen as shifting the whole tree in space by $\ell$. Moreover, $\Xi_\ell$ can be seen as a barrier operator: $\Xi_\ell(t)$ is the branching random walk $t$ where particles entering $(-\infty, \ell)$ are instantaneously killed; $\Xi(t)$ is thus the branching random walk with a barrier at the location of the root.

Further, we define $q = \P(|\Xi(T)| = +\infty)$ as the probability that the tree $\Xi(T)$ is infinite.

\begin{lemma}\label{lemma:barrier}
	$q > 0$.
\end{lemma}

\begin{proof}
	If $\E X > 0$, the result is clear since then any line of descent of $T$ is a random walk with positive drift. If $\E X < 0$, $\P(X > 0) > 0$ and $p > 1/(1+a)$, then the result is given by Theorem~$1$ in Biggins et al.~\cite{Biggins91:0} under the additional assumption that $\E(e^{\eta X}) < +\infty$ for some $\eta > 0$. We now show that the result remains valid when $\E(e^{\eta X}) = +\infty$ for every $\eta > 0$ by a truncation and coupling argument. Note that in this case we have $a =1$, see Remark~\ref{rem:a}.
	
	From $T$ construct the tree $T_K$ for $K \geq 0$ obtained by replacing each label on an edge greater than $K$ by $K$. Then $T_K$ is in distribution equal to $T$, but where labels on the edges are distributed according to $X_K = \min(X, K)$ instead of $X$. Let $a_K = \inf_\theta \E(e^{\theta X_K})$, so that $a_K < 1$ and $a_K \to a$ as $K \to +\infty$. Since $a = 1$, we thus have $p > 1/(1+a_K)$ for $K$ large enough, and since $\E(e^{X_K}) < +\infty$, we can invoke the first step of the proof to get that $\P(|\Xi_0(T_K)| = +\infty) > 0$ for $K$ large enough. Since nodes have smaller labels in $T_K$ than in $T$, this immediately implies $\P(|\Xi_0(T)| = +\infty) > 0$ as well which concludes the proof.
\end{proof}

For $t \in \Tcal$ and $v \in \Vcal(t)$, let $\theta(v,t)$ be the tree $t$ shifted at $v$, i.e., $\theta(v,t)$ is the subtree of $t$ rooted at $v$. The following observation will underly our reasoning, where we consider the process $(B_n)$ as an iteration of $\Phi$ on $T$ and in particular, we identify an order with the corresponding node in the tree. Imagine that a node $v \in \Vcal(T)$ is becoming the price at some time $n$, i.e., $n = \inf\{ k \geq 0:v = \gamma \circ \Phi_k(T) \}$, and that $v$ remains green up to time $n' \geq n$. Then for any $n \leq k < n'$, each time we iterate $\Phi$, a node is changed color. Each such node must belong to $\theta(v, T)$, and furthermore, if a node $v'$ becomes the price, it must belong to $\Xi \circ \theta(v,T)$. In particular, if $v$ becomes red, then all the nodes belonging to $\Xi \circ \theta(v,T)$ must also be red (note that this fact relies on our choice of $\gamma(t)$ being the largest node, in the lexicographic order, in case of several green nodes having the largest label).

With this in mind, we now proceed to proving that $\P(\pi(B_n) \to +\infty) = 1$: we break the proof in two steps.
\\

\noindent \textit{First step.} In the first step we prove that it is enough to prove that
\begin{equation} \label{eq:first-step}
	\P \left( \pi(B_n) \to +\infty \mid |\Xi(T)| = +\infty \right) = 1.
\end{equation}

Note that $\P(|\Xi(T)| = +\infty) > 0$ by the previous lemma so this conditioning is well-defined. Let a forest be a finite collection of trees in $\Tcal$. Consider a sequence of i.i.d.\ trees $(T_n)$ with common distribution $T$, and let $(\Ccal_n, 0 \leq n \leq \zeta)$ be the (possibly finite, if $\zeta < +\infty$) sequence of forests defined recursively as follows. At time $0$ we have $\Ccal_0 = \{T_0\}$.

Let $n \geq 0$: if $\Ccal_n = \emptyset$, then $\Ccal_{n+1} = \{T_{n+1}\}$. Else, we pick the tree, say $t_n$, in $\Ccal_n$ whose root has the largest label, say $\ell_n$, and we remove it from $\Ccal_n$ (ties are broken at random). If $|\Xi(t_n)| = +\infty$, we stop and define $\zeta = n$. Otherwise, we consider the (possibly empty, but finite) forest $t_n \setminus \Xi(t_n) = \{ t'_{n,1}, \ldots, t'_{n,K} \}$. We then obtain $\Ccal_{n+1}$ by the formula $\Ccal_{n+1} = (\Ccal_n \setminus \{t_n\}) \cup \{ t''_{n,1}, \ldots, t''_{n,K}\}$ where $t''_{n,k}$ is obtained from $t'_{n,k}$ by painting the root in green (note that, by construction, all other nodes are white).

It follows from this construction that the trees $(S_{-\ell_n}(t_n), 0 \leq n < \zeta)$ are identically distributed, with common distribution $T$ conditioned on $\{ |\Xi(T)| < +\infty \}$. Moreover, although these trees are in general not independent, due to the fact that $t_{n'}$ for $n' > n$ may be a subtree of $t_n$, the trees $(\Xi(S_{-\ell_n}(t_n)), 0 \leq n < \zeta)$ are independent. Indeed, consider for instance the tree $\Xi(S_{-\ell_1}(t_1))$ when $\zeta > 1$: then, although $t_1$ is a subtree of $t_0$, it depends on $\Xi(t_0)$ only through the label of its root. In particular, $S_{-\ell_1}(t_1)$ is independent from $\Xi(t_0)$. The  sequence $(\Xi(S_{-\ell_n}(t_n)), 0 \leq n < \zeta)$ is therefore an i.i.d.\ sequence of trees with common distribution $\Xi(T)$ conditioned on $\{ |\Xi(T)| < +\infty \}$.

It follows that $\zeta$ is a geometric random variable with parameter $q$, and since $q > 0$ by Lemma~\ref{lemma:barrier}, $\zeta$ is almost surely finite. We can therefore consider the tree $\Xi(S_{-\ell_\zeta}(t_\zeta))$, which, following similar arguments as before, is independent from the sequence of trees $(S_{-\ell_n}(t_n), 0 \leq n < \zeta)$, and is distributed like $\Xi(T)$ conditioned on $\{ |\Xi(T)| = +\infty \}$.

Furthermore, once we have the sequence of trees $(t_n, 0 \leq n \leq \zeta)$, we can iterate $\Phi$ on it: first, we iterate $\Phi$ on $t_0$ until all nodes of $\Xi(t_0)$ are red; then we proceed on iterating $\Phi$ on $t_1$ until all nodes of $\Xi(t_1)$ are red, etc\ldots\ Eventually, we will be iterating $\Phi$ on $t_\zeta$ and since $|\Xi(t_\zeta)| = +\infty$, there will always be green nodes in $\Xi(t_\zeta)$ in order to iterate $\Phi$. Moreover, it follows from Theorem~\ref{thm:coupling} that the process keeping track of the labels of all the green nodes either in the tree being explored or in the trees in the ``waiting room'' $\Ccal_n$ is precisely a version of the process $(B_n)$ started at $B_0 = \delta_0$.

Thus in order to study the long-term asymptotic behavior of the price, we may as well start right away with $T_0 = t_\zeta$, i.e., in order to prove $\P(\pi(B_n) \to +\infty) = 1$ it is enough to prove that $\P \left( \pi(B_n) \to +\infty \mid T = t_\zeta \right) = 1$. But if $\pi(B_n) \to +\infty$ almost surely in the event $T = t_\zeta$, shifting in space by $\ell_\zeta$ will not affect the result, and so it is indeed enough to prove~\eqref{eq:first-step}.
\\

\noindent \textit{Second step.} We now consider everything in the event $\{|\Xi(T)| = +\infty\}$ and prove~\eqref{eq:first-step} (recall that $\{|\Xi(T)| = +\infty\} = \{|\Xi_0(T)| = +\infty\}$ since by definition the root of $T$ has label $0$). Let $v_1, \cdots, v_{I_1}$ be the children of the root in $T$ ranked in lexicographic order. In $\{ |\Xi_0(T)| = +\infty \}$, we have $I_1 \geq 1$ and the set $\{ i: |\Xi_0 \circ \theta(v_i, T)| = +\infty \}$ is not empty. Let $i_1 = \min \{ i: |\Xi_0 \circ \theta(v_i, T)| = +\infty \}$ and $v^*_1 = v_{i_1}$, i.e., $v^*_1$ is the first child of the root that has an infinite line of descent that never enters $(-\infty, 0)$. Iterating this procedure, we can define a sequence $(v^*_n)$ such that $v^*_n$ is the first child of $v^*_{n-1}$ such that $|\Xi_0 \circ \theta(v^*_n, T)| = +\infty$. Note that the dynamic of $\Phi$ is such that eventually, every node $v^*_n$ becomes green. Moreover, $v^*_n$ stays green forever if and only if $|\Xi \circ \theta(v^*_n, T)| = +\infty$.

Let $u(k)$ be the index of the $k$th node of the sequence $(v^*_n)$ that stays green forever, i.e., $u(1) = \inf\{ k \geq 1: |\Xi \circ \theta(v^*_k, T)| = +\infty \}$ and for $n \geq 1$, $u(n+1) = +\infty$ if $u(n) = +\infty$ and otherwise,
\[ u(n+1) = \inf\{ k > u(n): |\Xi \circ \theta(v^*_k, T)| = +\infty \}. \]

Let $\ell^*_n = \ell(v^*_n, T)$ be the label of $v^*_n$: if $u(k)$ is finite, then $v^*_{u(k)}$ stays green forever and in particular, $\liminf_{n \to +\infty} \pi(B_n) \geq \ell^*_{u(k)}$. Thus if $u(n)$ is finite for each $n \geq 1$, we obtain, since the sequence $(\ell^*_{u(n)})$ is by construction non-decreasing and so admits a limit in $(-\infty,\infty]$,
\[ \liminf_{n \to +\infty} \pi(B_n) \geq \lim_{n \to +\infty} \ell^*_{u(n)}. \]

Conditionally on the event $\{ |\Xi_0(T)| = +\infty \text{ and } \forall n: u(n) < +\infty \}$, we see by shifting the tree successively at the nodes $v^*_{u(1)}$, $v^*_{u(2)}$, \ldots\ that the random variables $(\ell^*_{u(n+1)} - \ell^*_{u(n)}, n \geq 1)$ are i.i.d., non-negative and strictly positive with strictly positive probability, from which it follows that $\ell^*_{u(n)} \to +\infty$ as $n \to +\infty$. Thus to conclude the proof, it remains to show that for every $n \geq 1$,
\[ \P\left(u(n) < +\infty \mid |\Xi(T)| = +\infty \right) = 1. \]

By regeneration (i.e., by shifting at $v^*_{u(1)}$) it is enough to prove this result for $n = 1$, and so we have to prove that
\[ \P \left( E_1 \cap E_2 \cap \cdots \mid |\Xi(T)| = +\infty \right) = 0 \text{ where } E_n = \left\{ |\Xi \circ \theta(v^*_n, T)| < +\infty \right\}. \]

Let $w(1) = 1$ and $w(n+1) = \inf \big\{ k > w(n): \ell^*_k < \ell^*_{w(n)} \big\}$: then $w(2)$ is finite in $E_{w(1)}$ and more generally, $w(n+1)$ is finite in $E_{w(1)} \cap \cdots \cap E_{w(n)}$. In particular, defining
\[ r_N = \P \left( E_{w(1)} \cap E_{w(2)} \cap \cdots \cap E_{w(N)} \mid |\Xi(T)| = +\infty \right) \]
we have $\P \left( E_1 \cap E_2 \cap \cdots \mid |\Xi(T)| = +\infty \right) \leq r_N$. Moreover, by definition of $r_N$ it holds that
\[ r_N = r_{N-1} \left[ 1 - \P \left( |\Xi \circ \theta(v^*_{w(N)}, T)| = +\infty \mid E_{w(1)} \cap \cdots \cap E_{w(N-1)}, |\Xi(T)| = +\infty \right) \right]. \]

Let $\ell \geq 0$: conditionally on the event
\[ E_{w(1)} \cap \cdots \cap E_{w(N-1)} \cap \{ |\Xi(T)| = +\infty, \ell^*_{w(N)} = \ell \}, \]
$\theta(v^*_{w(N)}, T)$ is equal in distribution to $S_\ell(T)$ conditioned on $\{ |\Xi_0 \circ S_\ell(T)| = +\infty \}$ (note that this last event has probability at least $q > 0$). In particular,
\begin{multline*}
	\P \left( |\Xi \circ \theta(v^*_{w(N)}, T)| = +\infty \mid E_{w(1)} \cap \cdots \cap E_{w(N-1)}, |\Xi(T)| = +\infty, \ell^*_{w(N)} = \ell \right)\\
	= \P \left( |\Xi \circ S_\ell(T)| = +\infty \mid |\Xi_0 \circ S_\ell(T)| = +\infty \right)
\end{multline*}
and by shifting the trees by $-\ell$, we obtain
\[ \P \left( |\Xi \circ S_\ell(T)| = +\infty \mid |\Xi_0 \circ S_\ell(T)| = +\infty \right) = \P \left( |\Xi(T)| = +\infty \mid |\Xi_{-\ell}(T)| = +\infty \right). \]

If $T$ survives with a barrier at $0$, it certainly survives with a barrier at $-\ell$: in particular,
\[ \P \left( |\Xi(T)| = +\infty \mid |\Xi_{-\ell}(T)| = +\infty \right) = \frac{\P \left( |\Xi(T)| = +\infty \right)}{\P \left( |\Xi_{-\ell}(T)| = +\infty \right)} \geq q. \]

This shows that $r_N \leq r_{N-1} (1-q)$ and by induction, $r_N \leq (1-q)^{N-1}$. Since $q > 0$ and $r_N$ is an upper bound on $\P(u(1) = +\infty \mid |\Xi(T)| = +\infty)$, by letting $N \to +\infty$ we finally get the desired result $\P(u(1) = +\infty \mid |\Xi(T)| = +\infty) = 0$ which achieves the proof of Theorem~\ref{thm:price}.

\section{Proof of Theorem~\ref{thm:coupling}} 

\subsection{Study of an auxiliary tree-valued Markov chain} Theorem~\ref{thm:coupling} is very intuitive. Unfortunately, the rigorous proof involves quite a lot of formalism, since we need to go into the details of the tree dynamic induced by iterations of $\Phi$. In order to slightly reduce the notational burden, we will assume that $X$ is a discrete random variable; it is just a matter of formalism to extend the proof below to the general case. Let us introduce the following tree operators:
\begin{itemize}
	\item for $t \in \Tcal$, $\Upsilon(t)$ is the tree obtained from $t$ by deleting all white nodes;
	\item for $t \in \Tcal_\Omega = \{ t \in \Tcal: |\Gcal(t)| > 0 \text{ and } \omega(t) > 0 \}$, $\Omega(t)$ is the tree obtained from $t$ by turning the first white child of $\gamma(t)$ into a green node;
	\item for $t \in \Tcal^*$ and $\lambda \in \R$, $\Omega'(t, \lambda)$ is the tree obtained from $t$ by adding a green child to $\gamma(t)$ with label $\lambda$ on the corresponding edge;
	\item for $t \in \Tcal^*$, $\Psi(t)$ is the tree obtained from $t$ by turning $\gamma(t)$ into a red node.
\end{itemize}
It will also be convenient to introduce the following subsets of $\Tcal$:
\begin{itemize}
	\item $\Tcal_\Upsilon$ is the set of finite rooted trees with only green or red nodes;
	\item $\Tcal_\Psi = \{ t \in \Tcal: |\Gcal(t)| > 0 \text{ and } \omega(t) = 0 \}$;
	\item $\Tcal_0$ is the set of trees of which every node is white, except for the root which is green.
\end{itemize}

Note that $\Tcal = \{ t: |\Gcal(t)| = 0 \} \cup \Tcal_\Omega \cup \Tcal_\Psi$. Let $Y_n = \Upsilon(\Phi_n(T))$ with $T \in \Tcal_0$ as in Theorem~\ref{thm:coupling}. We denote by $t_0 = \Upsilon(T)$ the deterministic tree reduced to the root, which is green. The goal of this section is to prove that the process $(Y_n, n \geq 0)$ defines a Markov chain started at $t_0$ with the following dynamic: for any $n \geq 0$, any $y, y' \in \Tcal_\Upsilon$ (note that for any $t \in \Tcal_0$ and any $k \geq 0$, $\Upsilon(\Phi_k(t)) \in \Tcal_\Upsilon$) and any $x \in \R$,
\[ \P(Y_{n+1} = y' \mid Y_n = y) = \begin{cases}
	1 & \text{ if } |\Gcal(y)| = 0 \text{ and } y' = y,\\
	p \P(X=x) & \text{ if } |\Gcal(y)| > 0 \text{ and } y' = \Omega'(y, x),\\
	1-p & \text{ if } |\Gcal(y)| > 0 \text{ and } y' = \Psi(y),\\
	0 & \text{ otherwise.}
\end{cases} \]

To this end, we fix until the rest of this section $n \geq 0$, $x \in \R$ and $y_k \in \Tcal_\Upsilon$ for $k = 0, \ldots, n+1$, and we aim to prove that
\begin{multline} \label{eq:goal}
	\P\left( Y_k = y_k, 0 \leq k \leq n+1 \right) =\\
	\P\left( Y_k = y_k, 0 \leq k \leq n \right) \times \begin{cases}
		1 & \text{if } |\Gcal(y_n)| = 0 \text{ and } y_{n+1} = y_n,\\
		p \P(X=x) & \text{if } |\Gcal(y_n)| > 0 \text{ and } y_{n+1} = \Omega'(y_n, x),\\
		1-p & \text{if } |\Gcal(y_n)| > 0 \text{ and } y_{n+1} = \Psi(y_n)
	\end{cases}
\end{multline}
which will prove the Markov property of $(Y_n)$ with the prescribed dynamic. Let us study the dynamic of $(Y_n)$. According to the various definitions made, we have for any $t \in \Tcal$ that $\Phi(t)$ is equal to $t$ if $|\Gcal(t)| = 0$, to $\Omega(t)$ if $t \in \Tcal_\Omega$ and to $\Psi(t)$ if $t \in \Tcal_\Psi$, so that
\[ \Upsilon(\Phi(t)) = \begin{cases}
		\Upsilon(t) & \text{ if } |\Gcal(t)| = 0,\\
		\Upsilon(\Omega(t)) & \text{ if } t \in \Tcal_\Omega,\\
		\Upsilon(\Psi(t)) & \text{ if } t \in \Tcal_\Psi.
		\end{cases} \]

It is clear that if $t \in \Tcal_\Omega$, then $\Upsilon(\Omega(t)) = \Omega'(\Upsilon(t), \ell(t))$ with $\ell(t)$ the label on the edge between $\gamma(t)$ and its first white child, while if $t \in \Tcal_\Psi$, then $\Upsilon(\Psi(t)) = \Psi(\Upsilon(t))$. Thus the previous display can be rewritten as
\[ \Upsilon(\Phi(t)) = \begin{cases}
	\Upsilon(t) & \text{ if } |\Gcal(t)| = 0,\\
	\Omega'(\Upsilon(t), \ell(t)) & \text{ if } t \in \Tcal_\Omega,\\
	\Psi(\Upsilon(t)) & \text{ if } t \in \Tcal_\Psi
	\end{cases} \]
and since $\Phi_{n+1}(t) = \Phi(\Phi_n(t))$ we get
\[ \Upsilon(\Phi_{n+1}(t)) = \begin{cases}
	\Upsilon(\Phi_n(t)) & \text{ if } |\Gcal(\Phi_n(t))| = 0,\\
	\Omega'(\Upsilon(\Phi_n(t)), \ell(\Phi_n(t))) & \text{ if } \Phi_n(t) \in \Tcal_\Omega,\\
	\Psi(\Upsilon(\Phi_n(t))) & \text{ if } \Phi_n(t) \in \Tcal_\Psi.
	\end{cases} \]

Since $\Upsilon$ does not affect the colors of green nodes, we have $\Gcal(t) = \Gcal(\Upsilon(t))$ and in particular, $|\Gcal(\Phi_n(t))| = |\Gcal(\Upsilon(\Phi_n(t)))|$. Plugging in the definitions of $\Tcal_\Omega$ and $\Tcal_\Psi$, it follows that for any $n \geq 0$ and any $t \in \Tcal_0$, we have
\begin{multline} \label{eq:upsilon-phi}
	\Upsilon(\Phi_{n+1}(t)) =\\
	\begin{cases}
		\Upsilon(\Phi_n(t)) & \text{if } |\Gcal(\Upsilon(\Phi_n(t)))| = 0,\\
		\Omega'(\Upsilon(\Phi_n(t)), \ell(\Phi_n(t))) & \text{if } |\Gcal(\Upsilon(\Phi_n(t)))| > 0 \text{ and } \omega(\Phi_n(t)) > 0,\\
		\Psi(\Upsilon(\Phi_n(t))) & \text{if } |\Gcal(\Upsilon(\Phi_n(t)))| > 0 \text{ and } \omega(\Phi_n(t)) = 0.
	\end{cases}
\end{multline}

This last equation shows that $\Upsilon(\Phi_{n+1}(t))$ is almost entirely determined by $\Upsilon(\Phi_n(t))$, up to the knowledge (hidden by the action of $\Upsilon$) of whether $\gamma(\Phi_n(t))$ has at least one white child in $\Phi_n(t)$ or not, and the value $\ell(\Phi_n(t))$ of the corresponding edge. Further, since $Y_n = \Upsilon(\Phi_n(T))$,~\eqref{eq:upsilon-phi} leads to
\begin{multline*}
	\P\left( Y_k = y_k, 0 \leq k \leq n+1 \right) = \indicator{|\Gcal(y_n)| = 0, y_{n+1} = y_n} \P\left( Y_k = y_k, 0 \leq k \leq n \right)\\
	+ \indicator{|\Gcal(y_n)| > 0, y_{n+1} = \Omega'(y_n, x)} \P\left( \omega(\Phi_n(T)) > 0, \ell(\Phi_n(T)) = x, Y_k = y_k, 0 \leq k \leq n \right)\\
	+ \indicator{|\Gcal(y_n)| > 0, y_{n+1} = \Psi(y_n)} \P\left( \omega(\Phi_n(T)) = 0, Y_k = y_k, 0 \leq k \leq n \right)
\end{multline*}
and so to prove~\eqref{eq:goal}, we only have to show that if $|\Gcal(y_n)| > 0$, then
\begin{multline} \label{eq:goal-1}
	\P\left( \omega(\Phi_n(T)) > 0, \ell(\Phi_n(T)) = x, Y_k = y_k, 0 \leq k \leq n \right)\\
	= p \P(X=x) \P\left( Y_k = y_k, 0 \leq k \leq n \right).
\end{multline}

This property is quite intuitive: the history of $Y_k$ for $k \leq n$ does not give any information on the remaining number of white children of $\gamma(\Phi_n(T))$ in $\Phi_n(T)$, nor on the label on the edge between $\gamma(\Phi_n(T))$ and its first white child, if any. The fact that every node has a geometric number of offspring and that labels on the edges are i.i.d.\ should therefore imply~\eqref{eq:goal-1}. To formalize this intuition, we will prove the following result, from which one can readily deduce~\eqref{eq:goal-1}. For $t \in \Tcal$, let in the rest of the paper $\eta(v,t)$ be the number of children of the node $v \in \Vcal(t)$. Recall moreover that $\Rcal(t)$ stands for the set of red nodes of the tree $t \in \Tcal$.

\begin{proposition}
	If $|\Gcal(y_n)| > 0$ and $\P\left( Y_k = y_k, 0 \leq k \leq n \right) > 0$, then for any $t \in \Tcal_0$ we have
	\[ \Upsilon(\Phi_k(t)) = y_k, 0 \leq k \leq n \Longleftrightarrow y_n \subset t \text{ and } \eta(v, y_n) = \eta(v, t) \text{ for every } v \in \Rcal(y_n). \]
\end{proposition}

\begin{proof}
	For $t \in \Tcal$, let in the rest of the proof $\sigma(t) = |\Gcal(t)| + 2 |\Rcal(t)| - 1$, and recall that $\kappa(t) = \inf\{ n \geq 0: |\Gcal(\Phi_n(t))| = 0 \}$. It is clear from the definition of $\Phi$ that $\sigma(\Phi(t)) = \sigma(t) + \indicator{|\Gcal(t)| > 0}$. Since $\Phi_n(t) = t$ for any $n \geq 0$ if $|\Gcal(t)| = 0$ and $\sigma(t) = 0$ for $t \in \Tcal_0$, it follows that
	\begin{equation} \label{eq:sigma}
		\forall t \in \Tcal_0, \ \sigma(\Upsilon(\Phi_k(t))) = k \Longleftrightarrow k \leq \kappa(t).
	\end{equation}
	We break the proof of the proposition into two steps.
	\\

	\noindent \textit{First step.} Fix some $t \in \Tcal_0$ and $y \in \Tcal_\Upsilon$. The first step of the proof consists in proving that the following conditions are equivalent:
	\begin{enumerate}[(i)]
		\item \label{cond:i} $\Upsilon(\Phi_{\sigma(y)}(t)) = y$;
		\item \label{cond:ii} there exists $t' \in \Tcal_0$ such that $\Upsilon(\Phi_{\sigma(y)}(t')) = y$, $y \subset t$ and $\eta(v,y) = \eta(v, t)$ for every $v \in \Rcal(y)$.
	\end{enumerate}

	\noindent \textit{Proof of \eqref{cond:i}$\Rightarrow$\eqref{cond:ii}}. Assume that~\eqref{cond:i} holds, i.e., $\Upsilon(\Phi_{\sigma(y)}(t)) = y$: we want to prove~\eqref{cond:ii}. Then taking $t' = t$ gives the existence of the desired $t'$. Moreover, since $\Phi(a)$ does not change the genealogical structure of $a \in \Tcal$ and $\Upsilon(a)$ only truncates $a$, we have $\Upsilon(\Phi_n(t)) \subset t$ for any $n \geq 0$, in particular $y \subset t$. Then, consider any $v \in \Rcal(y)$. Since all the nodes of $t$ except for the root are white, the color of $v$ results from the successive applications of $\Phi$ to $t \in \Tcal_0$. In particular, $v$ being red in $\Phi_{\sigma(y)}(t)$ comes from the fact that at some point, none of the children of $v$ were white, i.e., $\omega(v, \Phi_k(t)) = 0$ for some $k \leq \sigma(y)$. Since $\Phi$ does not create white nodes, this implies $\omega(v, \Phi_{\sigma(y)}(t)) = 0$ and since $\Upsilon$ conserves all non-white nodes, $v$ has as many children in $\Phi_{\sigma(y)}(t)$ as in $\Upsilon(\Phi_{\sigma(y)}(t)) = y$, i.e., $\eta(v,\Phi_{\sigma(y)}(t)) = \eta(v,y)$ which gives $\eta(v,t) = \eta(v,y)$.
	\\

	\noindent \textit{Proof of \eqref{cond:ii}$\Rightarrow$\eqref{cond:i}}. Assume that~\eqref{cond:ii} holds: we want to prove~\eqref{cond:i}. So in the rest of the proof, consider some $t' \in \Tcal_0$ such that $\Upsilon(\Phi_{\sigma(y)}(t')) = y$, and assume that $y \subset t$ and $\eta(v,y) = \eta(v,t)$ for every $v \in \Rcal(y)$. We prove that $y = \Upsilon(\Phi_{\sigma(y)}(t))$ by induction on $\sigma(y)$.

	 If $\sigma(y) = 0$, then on the one hand, $\Upsilon(\Phi_{\sigma(y)}(t)) = t_0$ while on the other hand, $|\Gcal(y)| + 2 |\Rcal(y)| - 1 = 0$ implies $|\Gcal(y)| = 1$ and $|\Rcal(y)| = 0$, so that $y = t_0$. Thus $y = \Upsilon(\Phi_{\sigma(y)}(t))$ when $\sigma(y) = 0$, which initializes the induction.

	Assume now that $\sigma(y) \geq 1$. Then $\sigma(y) = \sigma(\Upsilon(\Phi_{\sigma(y)}(t'))$ and so~\eqref{eq:sigma} implies that $\sigma(y) \leq \kappa(t')$. Define $y' = \Upsilon(\Phi_{\sigma(y)-1}(t'))$: since $\sigma(y)-1 \leq \kappa(t')$,~\eqref{eq:sigma} implies that $\sigma(y') = \sigma(y)-1$, i.e., $\sigma(\Phi_{\sigma(y)}(t')) = 1 + \sigma(\Phi_{\sigma(y)-1}(t'))$. This last equality means that applying $\Phi$ on $\Phi_{\sigma(y)-1}(t')$ creates a green node or changes a green node into a red one, meaning in every case that $|\Gcal(\Phi_{\sigma(y)-1}(t'))| > 0$ and since $\Gcal(\Phi_{\sigma(y)-1}(t')) = \Gcal(\Upsilon(\Phi_{\sigma(y)-1}(t'))) = \Gcal(y')$ this finally means that $|\Gcal(y')| > 0$. In view of $y = \Upsilon(\Phi_{\sigma(y)}(t'))$, $y' = \Upsilon(\Phi_{\sigma(y)-1}(t'))$ and~\eqref{eq:upsilon-phi}, we therefore only have two possibilities:
	\begin{equation} \label{eq:y-y'}
		y = \begin{cases}
			\Omega'(y', \ell(\Phi_{\sigma(y)-1}(t'))) & \text{ if } \omega(\Phi_{\sigma(y)-1}(t')) > 0,\\
			\Psi(y') & \text{ if } \omega(\Phi_{\sigma(y)-1}(t')) = 0.
		\end{cases}
	\end{equation}

	In either case, we have $y' \subset y$ and since $y \subset t$ by assumption, this gives $y' \subset t$. Moreover, the action of $\Omega'$ is to add one green node to $\gamma$ and the action of $\Psi$ is to turn $\gamma$ into a red node, so that in either case we have $\Rcal(y') \subset \Rcal(y)$ and $\eta(v, y') = \eta(v, y)$ for every $v \in \Rcal(y')$. Since $\eta(v, y) = \eta(v, t)$ for every $v \in \Rcal(y)$ by assumption, this implies that $\eta(v,y') = \eta(v,t)$ for every $v \in \Rcal(y')$. Since finally $y' = \Upsilon(\Phi_{\sigma(y')}(t'))$ and $\sigma(y') = \sigma(y)-1 < \sigma(y)$, we can therefore invoke the induction hypothesis to deduce that $\Upsilon(\Phi_{\sigma(y')}(t)) = y' = \Upsilon(\Phi_{\sigma(y')}(t'))$. In particular, $\Phi_{\sigma(y')}(t)$ and $\Phi_{\sigma(y')}(t')$ have the same set of green and red nodes and since $|\Gcal(\Phi_{\sigma(y')}(t'))| > 0$ this shows that $|\Gcal(\Phi_{\sigma(y')}(t))| > 0$. Then,~\eqref{eq:upsilon-phi} shows that
	\[ \Upsilon(\Phi_{\sigma(y)}(t)) = \begin{cases}
		\Omega'(\Upsilon(\Phi_{\sigma(y')}(t)), \ell(\Phi_{\sigma(y')}(t))) & \text{if } \omega(\Phi_{\sigma(y')}(t)) > 0,\\
		\Psi(\Upsilon(\Phi_{\sigma(y')}(t))) & \text{if } \omega(\Phi_{\sigma(y')}(t)) = 0.
	\end{cases} \]
	
	Since $y' = \Upsilon(\Phi_{\sigma(y')}(t))$ this can be rewritten as
	\[ \Upsilon(\Phi_{\sigma(y)}(t)) = \begin{cases}
		\Omega'(y', \ell(\Phi_{\sigma(y')}(t))) & \text{if } \omega(\Phi_{\sigma(y')}(t)) > 0,\\
		\Psi(y') & \text{if } \omega(\Phi_{\sigma(y')}(t)) = 0,
	\end{cases} \]
	and in view of~\eqref{eq:y-y'}, the proof of $y = \Upsilon(\Phi_{\sigma(y)}(t))$ will be complete if we can show the two following implications:
	\[ \omega(\Phi_{\sigma(y')}(t')) > 0 \Rightarrow \omega(\Phi_{\sigma(y')}(t)) > 0 \text{ and } \ell(\Phi_{\sigma(y')}(t)) = \ell(\Phi_{\sigma(y')}(t')) \]
	and
	\[ \omega(\Phi_{\sigma(y')}(t')) = 0 \Rightarrow \omega(\Phi_{\sigma(y')}(t)) = 0. \]

	To prove these two implications, we will use the identities
	\begin{equation} \label{eq:w-1}
		\omega(\Phi_{\sigma(y')}(t')) = \eta(\gamma(y'), t') - \eta(\gamma(y'), y')
	\end{equation}
	and
	\begin{equation} \label{eq:w-2}
		\omega(\Phi_{\sigma(y')}(t)) = \eta(\gamma(y'), t) - \eta(\gamma(y'), y')
	\end{equation}
	that come from the fact that $\Upsilon(\Phi_{\sigma(y')}(t')) = \Upsilon(\Phi_{\sigma(y')}(t))$.

	\begin{description}
		\item[Direct implication] assume that $\omega(\Phi_{\sigma(y')}(t')) > 0$. In this case, applying $\Phi$ to $\Phi_{\sigma(y')}(t')$ adds a green child to $\gamma(y')$ in $\Phi_{\sigma(y)}(t')$, and so $\eta(\gamma(y'), y) = 1 + \eta(\gamma(y'), y')$. Since $y \subset t$ by assumption, this gives $\eta(\gamma(y'), t) > \eta(\gamma(y'), y')$ which proves that $\omega(\Phi_{\sigma(y')}(t)) > 0$ in view of~\eqref{eq:w-1}. Moreover, the equality $\eta(\gamma(y'), \Upsilon(\Phi_{\sigma(y')}(t'))) = \eta(\gamma(y'), \Upsilon(\Phi_{\sigma(y')}(t)))$ means that $\gamma(y')$ has as many green and red children in $\Phi_{\sigma(y')}(t')$ than in $\Phi_{\sigma(y')}(t)$. In particular, applying $\Phi$ to these two trees adds the same node, say $v$, to each tree. Then, the equality $\ell(\Phi_{\sigma(y')}(t)) = \ell(\Phi_{\sigma(y')}(t'))$ comes from the fact that the label on the edge between $\gamma(y')$ and $v$ is the same in $t$ and $t'$ due to the inclusion $y = \Upsilon(\Phi_{\sigma(y)}(t')) \subset t$;
		\item[Reverse implication] assume that $\omega(\Phi_{\sigma(y')}(t')) = 0$. In this case, applying $\Phi$ to $\Phi_{\sigma(y')}(t')$ turns $\gamma(y')$ into a red node in $\Phi_{\sigma(y)}(t')$, i.e., $\gamma(y') \in \Rcal(y)$ and in particular, $\eta(\gamma(y'), y) = \eta(\gamma(y'), t)$ by assumption on $y$. On the other hand, $\Phi$ did not change the genealogical structure of $y'$ and so $\eta(\gamma(y'), y) = \eta(\gamma(y'), y')$ which shows that $\eta(\gamma(y'),y') = \eta(\gamma(y'), t)$. This implies $\omega(\Phi_{\sigma(y')}(t)) = 0$ in view of~\eqref{eq:w-2}.
	\end{description}
	This concludes the proof of the first step.
	\\

	\noindent \textit{Second step.} We now prove the proposition. Assume in the rest of the proof that $|\Gcal(y_n)| > 0$ and $\P\left( Y_k = y_k, 0 \leq k \leq n \right) > 0$. The direct implication is rather straightforward: if $\Upsilon(\Phi_k(t)) = y_k$ for every $k = 0, \ldots, n$, then for $k = n$ the result of the first step implies that $y_n \subset t$ and $\eta(v, y_n) = \eta(v, t)$ for every $v \in \Rcal(y_n)$.
		\\

		Let us now prove the converse implication, so assume that $y_n \subset t$ and $\eta(v, y_n) = \eta(v, t)$ for every $v \in \Rcal(y_n)$. The goal is to prove that $\Upsilon(\Phi_k(t)) = y_k$ for every $k = 0, \ldots, n$. Since $\P\left(Y_k = y_k, 0 \leq k \leq n \right) > 0$, there exists $t' \in \Tcal_0$ such that $y_k = \Upsilon(\Phi_k(t'))$ for every $k = 0, \ldots, n$. Because $\Phi$ does not erase nodes and never changes the color of a red node, we have $y_k \subset y_{k+1}$ and $\Rcal(y_k) \subset \Rcal(y_{k+1})$ for any $0 \leq k \leq n-1$. In particular, for every $0 \leq k \leq n$ it holds that $y_k \subset t$ and $\eta(v, y_k) = \eta(v, t)$ for every $v \in \Rcal(y_k)$. Moreover, $\Phi$ cannot create green nodes starting from a tree with no green node, and so the condition $|\Gcal(y_n)| > 0$ implies that $|\Gcal(y_k)| > 0$ for every $k \leq n$. Thus $|\Gcal(\Phi_k(t'))| > 0$ and so $k < \kappa(t)$, so that~\eqref{eq:sigma} implies that $\sigma(y_k) = k$, i.e., $y_k = \Upsilon(\Phi_{\sigma(y_k)}(t'))$. Thus all the assumptions of the first step are satisfied, and we deduce that $y_k = \Upsilon(\Phi_{\sigma(y_k)}(t)) = \Upsilon(\Phi_k(t))$ for every $k \leq n$. This finally concludes the proof of the proposition.
\end{proof}

\subsection{Proof of Theorem~\ref{thm:coupling}}

By regeneration of $(B_n)$ at $\tau+1$ and in view of~\eqref{eq:transition}, it is enough to show that $(\Gamma(\Phi_n(T)), n \geq 0)$ is the Markov chain started at $\delta_0$ with the following transition: for any $\beta \in \Bcal$ and any measurable function $f: \Bcal \to [0,\infty)$,
\begin{multline*}
	\E\left[f\big(\Gamma(\Phi_{n+1}(T))\big) \mid \Gamma(\Phi_n(T)) = \beta \right] = f(\delta_0) \indicator{|\beta| = 0} + p \E \left[ f(\beta + \delta_{\pi(\beta) + X}) \right] \indicator{|\beta| > 0}\\
	+ (1-p) f(\beta - \delta_{\pi(\beta)}) \indicator{|\beta| > 0}.
\end{multline*}

Since $\Gamma$ only depends on the green nodes, we have $\Gamma(\Phi_n(T)) = \Gamma(Y_n)$ and so we only have to show that $\Gamma(Y_n)$ is the Markov chain with the prescribed dynamic. This is easily verified once one realizes that the $\sigma$-algebras $\sigma(Y_k, 0 \leq k \leq n)$ and $\sigma(\Gamma(Y_k), 0 \leq k \leq n)$ are the same. Indeed, the inclusion $\sigma(\Gamma(Y_k), 0 \leq k \leq n) \subset \sigma(Y_k, 0 \leq k \leq n)$ is trivial. For the reverse inclusion, note that from the sequence $(\Gamma(Y_k), 0 \leq k \leq n)$ one can recover the sequence $(Y_k, 0 \leq k \leq n)$. This can be proved by induction on $n$. For $n = 0$ this is trivial, since $Y_0 = t_0$ and $\Gamma(Y_0) = \delta_0$. So assume it holds for $n \geq 0$, and let us prove it for $n+1$. So assume that $(\Gamma(Y_k), 0 \leq k \leq n+1)$ is known. Then by induction hypothesis, $(Y_k, 0 \leq k \leq n)$ is known. Moreover,~\eqref{eq:upsilon-phi} shows that there are only three possible cases:
\begin{itemize}
	\item either $Y_{n+1} = Y_n$, in which case $\Gamma(Y_{n+1}) = \Gamma(Y_n)$;
	\item or $Y_{n+1} = \Omega'(Y_n, \lambda)$ for some $\lambda \in \R$: in this case, $\Gamma(Y_{n+1}) = \Gamma(Y_n) + \delta_{\lambda}$;
	\item or $Y_{n+1} = \Psi(Y_n)$, in which case $\Gamma(Y_{n+1}) = \Gamma(Y_n) - \delta_\lambda$ for some $\lambda \in \R$.
\end{itemize}

Thus by comparing $\Gamma(Y_{n+1})$ to $\Gamma(Y_n)$ one can recover $Y_{n+1}$ which shows that $\sigma(\Gamma(Y_k), 0 \leq k \leq n) \subset \sigma(Y_k, 0 \leq k \leq n)$. From this, one easily deduces using the Markov property of $(Y_n)$ that $(\Gamma(Y_n))$ is a Markov chain with the prescribed dynamic. The proof of Theorem~\ref{thm:coupling} is therefore complete.

\bibliographystyle{apt}

\end{document}